\documentclass{amsart}

\usepackage[english]{babel}
\usepackage{amsfonts}
\usepackage{amsthm}
\usepackage{amssymb}
\usepackage{mathrsfs}

\input xy
\xyoption{all}

\newtheorem{defi}{Definition}
\newtheorem{teor}{Theorem}
\newtheorem{cor}{Corollary}
\newtheorem{prop}{Proposition}
\newtheorem{rem}{Remark}
\newtheorem{lemma}{Lemma}
\newtheorem{conj}{Conjecture}
\newtheorem{nota}{Notations}

\DeclareMathOperator{\Sym}{Sym}
\DeclareMathOperator{\Aut}{Aut}
\DeclareMathOperator{\Alt}{Alt}
\DeclareMathOperator{\soc}{soc}

\begin{document}

\title{Covering monolithic groups \\ with proper subgroups}
\author{Martino Garonzi}

\begin{abstract}
Given a finite non-cyclic group $G$, call $\sigma(G)$ the smallest number of proper subgroups of $G$ needed to cover $G$. Lucchini and Detomi conjectured that if a nonabelian group $G$ is such that $\sigma(G) < \sigma(G/N)$ for every non-trivial normal subgroup $N$ of $G$ then $G$ is \textit{monolithic}, meaning that it admits a unique minimal normal subgroup. In this paper we show how this conjecture can be attacked by the direct study of monolithic groups.
\end{abstract}

\maketitle

Every group considered in this paper is assumed to be finite, unless specified otherwise.

Given a non-cyclic group $G$, call $\sigma(G)$ - the \textit{covering number} of $G$ - the smallest number of proper subgroups of $G$ whose union equals $G$. It is an easy exercise to show that $\sigma(G) > 2$ (i.e. no group is the union of two proper subgroups). Note that there always exist minimal covers consisting of maximal subgroups. The covering number has been introduced the first time by Cohn in 1994 \cite{cohn}. We usually call \textit{cover} of $G$ a family of proper subgroups of $G$ which covers $G$, and \textit{minimal cover} of $G$ a cover of $G$ consisting of exactly $\sigma(G)$ elements. If $G$ is cyclic then $\sigma(G)$ is not well defined because no proper subgroup contains any generator of $G$; in this case we define $\sigma(G)=\infty$, with the convention that $n < \infty$ for every integer $n$.

\begin{rem}
If $N$ is a normal subgroup of a group $G$ then $\sigma(G) \leq \sigma(G/N)$: indeed, every cover of $G/N$ can be lifted to a cover of $G$.
\end{rem}

Given a family $\mathcal{H}$ of subsets of a group $G$ which covers $G$, we say that $\mathcal{H}$ is ``irredundant'' if $\bigcup_{\mathcal{H} \ni K \neq H} K \neq G$ for every $H \in \mathcal{H}$. Clearly every minimal cover is irredundant, but the converse is false. Actually the notion of irredundant cover is much weaker than that of minimal cover: for example, if $n \geq 2$ is an integer then the cover of ${C_2}^n$ consisting of its non-trivial cyclic subgroups is irredundant of size $2^n-1$ while ${C_2}^n$ has an epimorphic image isomorphic to $C_2 \times C_2$ so $\sigma({C_2}^n) = 3$.

We are interested in groups with finite covering number. The following result implies that in order to study the behaviour of the function which assigns to each group with finite covering number its covering number it is enough to consider finite groups.

\begin{teor}[Neumann 1954] \label{neum}
Let $G$ be an infinite group covered by a finite family $\mathcal{H}$ of cosets of subgroups of $G$, and suppose that $\mathcal{H}$ is irredundant. Then every $H \in \mathcal{H}$ has finite index in $G$.
\end{teor}

\begin{proof}
For a proof see Lemma 4.17 in \cite{neum}.
\end{proof}

Indeed, if $\mathcal{H}$ is a minimal cover of $G$ then by Theorem \ref{neum} $\bigcap_{H \in \mathcal{H}}H$ has finite index in $G$, hence its normal core $N$ has also finite index and $$\sigma(G/N) \leq |\mathcal{H}| = \sigma(G) \leq \sigma(G/N),$$thus $\sigma(G)=\sigma(G/N)$. In other words we are reduced to consider the covering number of the finite group $G/N$.

The solvable groups were studied by Tomkinson. He proved the following result. Recall that a ``chief factor'' of a group $G$ is a minimal normal subgroup $H/K$ of a quotient $G/K$ of $G$.

\begin{teor}[Tomkinson] \label{tom}
If $G$ is a finite non-cyclic solvable group then $\sigma(G)=q+1$ where $q$ is the order of the smallest chief factor $H/K$ of $G$ with more than one complement in $G/K$.
\end{teor}

Note that the number $q$ in the statement of Theorem \ref{tom} is a prime power. Not every $\sigma(G)$ is of the form $q+1$ with $q$ a prime power, for example $\sigma(\Sym(6))=13$ (cfr. \cite{shaker}).

Assume we want to compute the covering number of a group $G$. If there exists $N \unlhd G$ with $\sigma(G) = \sigma(G/N)$ then we may consider as well the quotient $G/N$ instead of $G$. This leads instantly to the following definition.

\begin{defi}[$\sigma$-elementary groups] \index{$\sigma$-elementary group}
We say that a group $G$ is ``$\sigma$-elementary'' if $\sigma(G) < \sigma(G/N)$ for every non-trivial normal subgroup $N$ of $G$.
\end{defi}

Clearly, every group has a $\sigma$-elementary quotient with the same covering number. It follows that the structure of the $\sigma$-elementary groups is of big interest. It was studied by Lucchini and Detomi in \cite{cubo}. They conjectured that:

\begin{conj} \label{mainconj}
Every non-abelian $\sigma$-elementary group is monolithic.
\end{conj}

Here a group is said to be ``\textit{monolithic}'' if it admits exactly one minimal normal subgroup.

\section{Covering nilpotent groups}

In this section we will compute the covering number of nilpotent groups in order to get the reader familiarized with the methods.

Let $p$ be a prime. Observe that the group $C_p \times C_p$ admits exactly $p+1$ proper subgroups, and all these subgroups are cyclic of order $p$ and index $p$. Let us visualize this in the subgroup lattice: $$
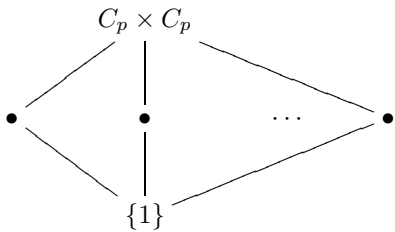
\xymatrix{& C_p \times C_p & & \\ \bullet \ar@{-}[ur] & \bullet \ar@{-}[u] & \cdots & \bullet \ar@{-}[ull] \\ & \{1\} \ar@{-}[ul] \ar@{-}[u] \ar@{-}[urr] & & }$$Therefore there is a unique cover of $C_p \times C_p$, it is the one consisting of all of its non-trivial proper subgroups. We obtain that $\sigma(C_p \times C_p) = p+1$.

\ 

The following result (which generalizes the equality $\sigma(C_p \times C_p) = p+1$) is a direct consequence of Theorem \ref{tom}. However, we will prove it in detail.

\begin{prop} \label{nilpotent}
Let $G$ be a finite nilpotent group. Then $\sigma(G) = p+1$ where $p$ is the smallest prime divisor of $|G|$ such that the Sylow $p$-subgroup of $G$ is not cyclic.
\end{prop}

Let us first observe that if $G$ is any finite group and $\Phi(G)$ is the Frattini subgroup of $G$ (i.e. the intersection of the maximal subgroups of $G$) then $\sigma(G) = \sigma(G/\Phi(G))$. Indeed, in any minimal cover of $G$ consisting of maximal subgroups its members all contain the Frattini subgroup.

Now suppose $G$ is a non-cyclic $p$-group. It is well known that $G/\Phi(G) \cong {C_p}^d$ where $d$ is the smallest size of a subset of $G$ generating $G$. Therefore $\sigma(G) = \sigma({C_p}^d)$. The covering number of ${C_p}^d$ can be easily computed using the following basic lemma.

\begin{lemma}[Minimal Index Lower Bound] \label{minind}
Let $\mathcal{H}$ be a minimal cover of a finite group $T$. Then $$\min \{|T:H|\ :\ H \in \mathcal{H}\} < \sigma(T).$$
\end{lemma}

\begin{proof}
Write $\mathcal{H} = \{H_1,\ldots,H_k\}$, $k=\sigma(T)$, $\beta_i := |T:H_i|$ with $\beta_1 \leq \cdots \leq \beta_k$. Since the union $H_1 \cup \cdots \cup H_k$ is not disjoint (because $1 \in H_i$ for $i=1,\ldots,k$), we have $$|T| = |\bigcup_{i=1}^k H_i| < \sum_{i=1}^k |H_i| = \sum_{i=1}^k |T|/\beta_i \leq k|T|/\beta_1.$$It follows that $\beta_1 < k = \sigma(T)$.
\end{proof}

Lemma \ref{minind} implies that $\sigma({C_p}^d) > p$. On the other hand, since $d > 1$ (because $G$ is non-cyclic), ${C_p}^d$ projects onto ${C_p}^2 = C_p \times C_p$, therefore $p < \sigma({C_p}^d) \leq \sigma(C_p \times C_p) = p+1$. We deduce that $\sigma(G) = \sigma({C_p}^d) = p+1$. Since any finite nilpotent group is the direct product of its Sylow subgroups, Proposition \ref{nilpotent} follows from the following lemma.

\begin{lemma}
Let $A,B$ be two finite groups of coprime order. Then $$\sigma(A \times B) = \min \{\sigma(A),\sigma(B)\}.$$
\end{lemma}

\begin{proof}
Let $\pi_A: A \times B \to A$, $\pi_B: A \times B \to B$ be the canonical projections. Let $\mathcal{H}$ be a minimal cover of $A \times B$ consisting of maximal subgroups, and let $$\Omega_A := \{H \in \mathcal{H}\ :\ \pi_B(H)=B\}, \hspace{1cm} \Omega_B := \{H \in \mathcal{H}\ :\ \pi_A(H)=A\}.$$Since $|A|,|B|$ are coprime, any subgroup of $A \times B$ is of the form $C \times D$ with $C \leq A$ and $D \leq B$. It follows that $\mathcal{H} = \Omega_A \cup \Omega_B$. Let $$O_A := A - \bigcup_{C \times B \in \Omega_A} C, \hspace{1cm} O_B := B - \bigcup_{A \times D \in \Omega_B} D.$$Since $\mathcal{H}$ covers $A \times B$, it covers $O_A \times O_B$, so $O_A \times O_B = \emptyset$. Hence, either $O_A = \emptyset$, implying $\Omega_B = \emptyset$ by minimality of $\mathcal{H}$ and $\sigma(A \times B) = \sigma(A)$, or $O_B = \emptyset$, implying $\Omega_A = \emptyset$ by minimality of $\mathcal{H}$ and $\sigma(A \times B) = \sigma(B)$.
\end{proof}

\section{Direct products of groups}

The very first case to consider when dealing with Conjecture \ref{mainconj} is the direct product case. In a joint work with A. Lucchini we deal with this case. We prove

\begin{teor}[Lucchini A., Garonzi M. 2010 \cite{garluc}]
Let $\mathcal{M}$ be a minimal cover of  a direct product $G=H_1 \times H_2$ of two groups.
Then one of the following holds:
\begin{enumerate}
\item  $\mathcal{M}=\{X\times H_2\mid X\in \mathcal{X}\}$ where $\mathcal{X}$ is a minimal cover of $H_1.$ In this case  $\sigma(G)=\sigma(H_1).$
\item  $\mathcal{M}=\{H_1\times X\mid X\in \mathcal{X}\}$ where $\mathcal{X}$ is a minimal cover of $H_2.$ In this case  $\sigma(G)=\sigma(H_2).$
\item  There exist $N_1\trianglelefteq H_1,$ $N_2\trianglelefteq H_2$ with $H_1/N_1\cong H_2/N_2\cong C_p$
and $\mathcal{M}$ consists of the maximal subgroups of $H_1\times H_2$ containing $N_1\times N_2.$ In this case $\sigma(G)=p+1.$
\end{enumerate}
\end{teor}

We will now give the idea of how the proof goes when $H_1$ and $H_2$ are isomorphic non-abelian simple groups. This does not cover all the ideas of the proof but it covers quite well those used when $H_1$ and $H_2$ do not have common abelian factor groups.

\ 

Let $S$ be a non-abelian simple group. We want to prove that $\sigma(S \times S) = \sigma(S)$. Note that since $S$ is a quotient of $S \times S$, $\sigma(S \times S) \leq \sigma(S)$.

\begin{enumerate}

\item We know that the maximal subgroups of $S \times S$ are of the following three types: $$(1)\ K \times S, \hspace{1cm} (2)\ S \times K, \hspace{1cm} (3)\ \Delta_{\varphi}:=\{(x,\varphi(x))\ |\ x \in S\},$$where $K$ is a maximal subgroup of $S$ and $\varphi \in \text{Aut}(S)$.

\item Let $\mathcal{M} = \mathcal{M}_1 \cup \mathcal{M}_2 \cup \mathcal{M}_3$ be a minimal cover of $S \times S$, where $\mathcal{M}_i$ consists of subgroups of type $(i)$.

\item Let $\Omega := S \times S - \bigcup_{M \in \mathcal{M}_1 \cup \mathcal{M}_2} M = \Omega_1 \times \Omega_2$, where $\Omega_1 = S - \bigcup_{K \times S \in \mathcal{M}_1} K$ and $\Omega_2 = S - \bigcup_{S \times K \in \mathcal{M}_2} K$.

\item We claim that it is enough to prove that $\Omega = \emptyset$. Indeed if this is the case then either $\Omega_1 = \emptyset$, in which case $\bigcup_{K \times S \in \mathcal{M}_1} K = S$ and $\mathcal{M} = \mathcal{M}_1$ by minimality of $\mathcal{M}$, or $\Omega_2 = \emptyset$, in which case $\bigcup_{S \times K \in \mathcal{M}_2} K = S$, and $\mathcal{M} = \mathcal{M}_2$ by minimality of $\mathcal{M}$. In both cases we obtain $\sigma(S \times S) \geq \sigma(S)$ and hence $\sigma(S \times S) = \sigma(S)$. \\ Suppose by contradiction $\Omega \neq \emptyset$, i.e. $\Omega_1 \neq \emptyset \neq \Omega_2$, and let $\omega \in \Omega_1$.

\item The family $$\{ K<S\ |\ S \times K \in \mathcal{M}_2\} \cup \{\langle \varphi(\omega) \rangle\ |\ \Delta_{\varphi} \in \mathcal{M}_3\}$$ is a cover of $S$ of size $|\mathcal{M}_2|+|\mathcal{M}_3|$ (it consists of proper subgroups being $S$ non-abelian). Indeed, if $b \in S$ is such that $b \not \in K$ for any $K < S$ such that $S \times K \in \mathcal{M}_2$ then $(\omega,b) \in S \times S - \Omega_1 \times \Omega_2$ hence, being $\mathcal{M}$ a cover for $S \times S$, $(\omega,b) \in \Delta_{\varphi}$ for some $\varphi \in \Aut(S)$ such that $\Delta_{\varphi} \in \mathcal{M}_3$, and we conclude that $b = \varphi(\omega) \in \langle \varphi(\omega) \rangle$.

\item It follows that $$|\mathcal{M}_1| + |\mathcal{M}_2| + |\mathcal{M}_3| = |\mathcal{M}| = \sigma(S \times S) \leq \sigma(S) \leq |\mathcal{M}_2| + |\mathcal{M}_3|.$$ This implies that $\mathcal{M}_1=\emptyset$. Analogously $\mathcal{M}_2=\emptyset$. So $\mathcal{M}=\mathcal{M}_3$.

\item Observe that since $S$ is covered by its non-trivial cyclic subgroups, $\sigma(S) < |S|$. Since each member of $\mathcal{M}_3 = \mathcal{M}$ has index $|S|$, by the Minimal Index Lower Bound (Lemma \ref{minind}) $$|S| < \sigma(S \times S) \leq \sigma(S) < |S|,$$ a contradiction.

\end{enumerate}

\section{Sigma star}

Recall that a group $G$ is called ``primitive'' if it admits a core-free maximal subgroup, that is, a maximal subgroup $M$ such that $\bigcap_{g \in G} gMg^{-1} = \{1\}$. A primitive group has always at most two minimal normal subgroup, and if they are two, they are non-abelian.

Recall that a $G$-group is a group $A$ endowed with a homomorphism $f: G \to \text{Aut}(A)$. If $a \in A$ and $g \in G$, the element $f(g)(a)$ is usually denoted $a^g$ if no ambiguity is possible.

\begin{defi}
Let $G$ be a group, and let $A,B$ be two $G$-groups.
\begin{itemize}
\item $A,B$ are said to be $G$-isomorphic \index{$G$-isomorphic $G$-groups} (written $A \cong_G B$) if there exists an isomorphism $\varphi:A \to B$ such that $a^{\varphi g} = a^{g \varphi}$ for every $g \in G$.
\item $A,B$ are said to be $G$-equivalent \index{$G$-equivalent $G$-groups} (written $A \sim_G B$) if there exist isomorphisms $$\xymatrix{\varphi:A \ar[r] & B},\ \xymatrix{\Phi:G \ltimes A \ar[r] & G \ltimes B}$$ such that the following diagram commutes: $$\xymatrix{\{1\} \ar[r] & A \ar[r] \ar[d]^-{\varphi} & G \ltimes A \ar[r] \ar[d]^-{\Phi} & G \ar[r] \ar@{=}[d] & \{1\} \\ \{1\} \ar[r] & B \ar[r] & G \ltimes B \ar[r] & G \ar[r] & \{1\}}$$
\end{itemize}
\end{defi}

Let $N$ be a minimal normal subgroup of a group $G$. The conjugation action of $G$ on $N$ gives $N$ the structure of $G$-group. Define $I_G(N)$ to be the set of elements of $G$ which induce by conjugation an inner automorphism of $N$ and define $R_G(N)$ to be the intersection of the normal subgroups $K$ of $G$ contained in $I_G(N)$ with the property that $I_G(N)/K$ is non-Frattini (i.e. not contained in the Frattini subgroup of $G/K$) and $G$-equivalent to $N$.

Recall that the ``socle'' of a group $G$, denoted $\soc(G)$, is the subgroup of $G$ generated by the minimal normal subgroups of $G$. $\soc(G)$ is always a direct product of some minimal normal subgroups of $G$. $G$ is said to be ``monolithic'' if it admits a unique minimal normal subgroup, i.e. if $\soc(G)$ is a minimal normal subgroup of $G$.

\begin{teor}[Lucchini, Detomi \cite{cubo} Corollary 14] \label{struttura}
Let $H$ be a non-abelian $\sigma$-elementary group and let $N_1,\ldots,N_{\ell}$ be minimal normal subgroups of $H$ such that $\soc(H) = N_1 \times \cdots \times N_{\ell}$. Let $X_i := G/R_H(N_i)$ for $i=1,\ldots,\ell$. Then $X_i$ is a primitive monolithic group with socle isomorphic to $N_i$ for $i=1,\ldots,\ell$ ($X_i$ will be called ``\textit{the primitive monolithic group associated to} $N_i$'') and $H$ is a subdirect product of $X_1, \ldots, X_{\ell}$: the canonical homomorphism $$H \to X_1 \times \ldots \times X_{\ell}$$is injective.
\end{teor}

\begin{defi}[Sigma star] \label{star}
Let $X$ be a primitive monolithic group, and let $N$ be its unique minimal normal subgroup. If $\Omega$ is an arbitrary union of cosets of $N$ in $X$ define $\sigma_{\Omega}(X)$ to be the smallest number of supplements of $N$ in $X$ needed to cover $\Omega$. If $\Omega = \{Nx\}$ we will write $\sigma_{Nx}(X)$ instead of $\sigma_{\{Nx\}}(X)$. Define $$\sigma^{\ast}(X) := \min \{\sigma_{\Omega}(X)\ |\ \Omega = \bigcup_i N \omega_i,\ \langle \Omega \rangle = X \}.$$
\end{defi}

\begin{prop}[Lucchini, Detomi \cite{cubo} Proposition 16] \label{sigmastar}
Let $H$ be a non-abelian $\sigma$-elementary group with socle $N_1 \times \cdots \times N_{\ell}$, $$H \leq_{\text{subd}} X_1 \times \ldots \times X_{\ell}$$as in Theorem \ref{struttura}. For $i=1,\ldots,\ell$ let $\ell_{X_i}(N_i)$ be the smallest primitivity degree of $X_i$, i.e. the smallest index of a proper supplement of $N_i$ in $X_i$. Then $\ell_{X_i}(N_i) \leq \sigma^{\ast}(X_i)$ for $i=1,\ldots,\ell$ and $$\sum_{i=1}^{\ell} \ell_{X_i}(N_i) \leq \sum_{i=1}^{\ell} \sigma^{\ast}(X_i) \leq \sigma(H).$$
\end{prop}

\begin{prop}[\cite{cubo}, Proposition 10] \label{abmns}
Let $G$ be a finite group. If $V$ is a complemented normal abelian subgroup of $G$ and $V \cap Z(G) = \{1\}$ then $\sigma(G) \leq 2|V|-1$.
\end{prop}

\begin{proof}
Let $H$ be a complement of $V$ in $G$. The idea is to show that $G$ is covered by the family $\{H^v\ |\ v \in V\} \cup \{C_H(v)V\ |\ 1 \neq v \in V\}$. We omit the details.
\end{proof}

\section{Small covering numbers}

The content of this section is included in my Ph.D. thesis.

\begin{lemma} \label{spancoprime}
Let $N$ be a normal subgroup of a group $X$. If a set of subgroups of $X$ covers a coset $yN$ of $N$ in $X$, then it also covers every coset $y^{\alpha}N$ with $\alpha$ prime to $|y|$.
\end{lemma}

\begin{proof}
Let $s$ be an integer such that $s \alpha \equiv 1 \mod |y|$. As $s$ is prime to $|y|$, by Dirichlet's theorem there exist infinitely many primes in the arithmetic progression $\{s+|y|n\ |\ n \in \mathbb{N}\}$; we choose a prime $p > |X|$ in $\{s + |y|n\ |\ n \in \mathbb{N}\}$. Clearly, $y^p=y^s$. As $p$ is prime to $|X|$, there exists an integer $r$ such that $pr \equiv 1 \mod |X|$. Hence, if $yN \subseteq \cup_{i \in I} M_i$, for every $g \in y^{\alpha}N$ we have that $g^p \in (y^{\alpha})^p N = (y^{\alpha})^s N = yN \subseteq \cup_{i \in I} M_i$ and therefore also $g=(g^p)^r$ belongs to $\cup_{i \in I}M_i$.
\end{proof}

\begin{prop} \label{56}
Let $H$ be a non-abelian $\sigma$-elementary group such that $\sigma(H) \leq 55$. Then $H$ is primitive and monolithic.
\end{prop}

\begin{proof}

We will use the notations of Theorem \ref{struttura}.

It is proven in \cite{cubo} that any non-abelian $\sigma$-elementary group has at most one abelian minimal normal subgroup. Therefore we may assume that there exists a non-abelian minimal normal subgroup $N$ of $H$. Let $G$ be the primitive monolithic group associated to $N$. If $G$ has a primitivity degree at most $27$ then either $\ell_G(N) \geq 10$ and $G/N \in \{C_2 \times C_2,\Sym(3),D_8\}$ (by inspection) - contradicting the inequality $\ell_G(N) \leq \sigma(H) \leq \sigma(G)$ (being $\sigma(C_2 \times C_2) = \sigma(D_8) = 3$ and $\sigma(S_3)=4$) - or $G/N$ is cyclic of prime-power order. Assume the latter case holds. Then $G/N$ admits only one maximal subgroup. In other words, a subset of $G$ generates $G$ modulo $N$ if and only if it contains an element $g \in G$ such that $G/N = \langle gN \rangle$. Thus Lemma \ref{spancoprime} implies that $\sigma(G) \leq \sigma^{\ast}(G)+1$, so that $$\sigma^{\ast}(X_1) + \sigma^{\ast}(X_2) \leq \sigma(H) \leq \sigma(X_1) \leq \sigma^{\ast}(X_1) + 1.$$In particular $\ell_{X_2}(N_2) \leq \sigma^{\ast}(X_2) \leq 1$, and this is a contradiction ($\ell_{X_2}(N_2)$ is the index of a proper subgroup of $X_2$).

Therefore we may assume that $\ell_G(N) \geq 28$ whenever $N$ is a non-abelian minimal normal subgroup of $G$. Suppose $H$ has at least two minimal normal subgroups $N_1=N, N_2$. If $N_2$ is non-abelian then by assumption $\ell_{X_2}(N_2) \geq 28$ and Proposition \ref{sigmastar} implies $56 \leq \ell_{X_1}(N_1) + \ell_{X_2}(N_2) \leq \sigma(H)$, a contradiction. Hence $N_2$ is abelian. We have $\ell_{X_2}(N_2) = |N_2|$ and by Proposition \ref{sigmastar} and Proposition \ref{abmns} $$28 + |N_2| \leq \ell_{X_1}(N_1) + \ell_{X_2}(N_2) \leq \sigma(H) \leq \sigma(X_2) < 2 |N_2|,$$ therefore $\sigma(H)-28 \geq |N_2| > \frac{1}{2} \sigma(H)$, and this implies $\sigma(H) > 56$, a contradiction.

\end{proof}

Proposition \ref{56} allows us to list the $\sigma$-elementary groups with small covering number. Indeed, if $H$ is a $\sigma$-elementary group such that $\sigma(H) \leq 55$ then $H$ is a primitive monolithic group with a primitivity degree at most $55$ (cf. Proposition \ref{sigmastar}). Since there are only finitely many groups of a given primitivity degree, we are reduced to look at a finite list of groups. By giving bounds to their covering numbers we can list the $\sigma$-elementary groups $G$ with $\sigma(G) \leq 25$. The explicit bounds can be found in \cite{gar}.

\begin{table}
\label{table25}
\begin{tabular}{|c|c|}
\hline $\sigma$ & $\text{Groups}$ \\ \hline $3$ & $C_2 \times C_2$ \\ \hline $4$ & $C_3 \times C_3,\Sym(3)$ \\ \hline $5$ & $\Alt(4)$ \\ \hline $6$ & $C_5 \times C_5,D_{10},AGL(1,5)$ \\ \hline $7$ & $\emptyset$ \\ \hline $8$ & $C_7 \times C_7,D_{14},7:3,AGL(1,7)$ \\ \hline $9$ & $AGL(1,8)$ \\ \hline $10$ & $3^2:4,AGL(1,9),\Alt(5)$ \\ \hline $11$ & $\emptyset$ \\ \hline $12$ & $C_{11} \times C_{11},11:5,D_{22},AGL(1,11)$ \\ \hline $13$ & $\Sym(6)$ \\ \hline $14$ & $C_{13} \times C_{13},D_{26},13:3,13:4,13:6,AGL(1,13)$ \\ \hline $15$ & $SL(3,2)$ \\ \hline $16$ & $\Sym(5),\Alt(6)$ \\ \hline $17$ & $2^4:5,AGL(1,16)$ \\ \hline $18$ & $C_{17} \times C_{17},D_{34},17:4,17:8,AGL(1,17)$ \\ \hline $19$ & $\emptyset$ \\ \hline $20$ & $C_{19} \times C_{19},AGL(1,19),D_{38},19:3,19:6,19:9$ \\ \hline $21$ & $\emptyset$ \\ \hline $22$ & $\emptyset$ \\ \hline $23$ & $M_{11}$ \\ \hline $24$ & $C_{23} \times C_{23},D_{46},23:11,AGL(1,23)$ \\ \hline $25$ & $\emptyset$ \\ \hline
\end{tabular}
\caption{The list of $\sigma$-elementary groups $G$ with $3 \leq \sigma(G) \leq 25$.}
\end{table}

In general, the following fact holds.

\begin{prop}
For every fixed positive integer $n$, the set of $\sigma$-elementary groups $H$ with $\sigma(H) = n$ is finite, bounded by a function of $n$.
\end{prop}

\begin{proof}
We will use the notations of Theorem \ref{struttura}. Let $H$ be a $\sigma$-elementary group, and write $\soc(H) = N_1 \times \ldots \times N_{\ell}$. Let $X_1,\ldots,X_{\ell}$ be the primitive monolithic groups associated to $N_1,\ldots,N_{\ell}$ respectively. $H$ embeds in $X_1 \times \ldots \times X_{\ell}$, so in order to conclude it suffices to bound the number of possibilities for $\ell$ and each $X_i$ in terms of $\sigma(H)$. By Proposition \ref{sigmastar} $$\ell \leq \sum_{i=1}^{\ell} \ell_{X_i}(N_i) \leq \sum_{i=1}^{\ell} \sigma^{\ast}(X_i) \leq \sigma(H).$$Since there are finitely many primitive groups with a given primitivity degree, the result follows.
\end{proof}

\section{Considering some monolithic groups}

The content of this section is included in my Ph.D. thesis.

Proposition \ref{56} holds also for $56$, but for this number a quite different argument is needed. This is interesting because of the following result, which is \cite[Theorem 2]{gar2}. Here $A_5 \wr C_2$ denotes the wreath product of $A_5$ with $C_2$, i.e. the semidirect product $(A_5 \times A_5) \rtimes C_2$ with the action of $C_2 = \langle \varepsilon \rangle$ on $A_5 \times A_5$ given by $(x,y)^{\varepsilon} = (y,x)$.

\begin{teor}[\cite{gar2} Theorem 2] \label{a5wrc2}
$\sigma(A_5 \wr C_2) = 1 + 4 \cdot 5 + 6 \cdot 6 = 57$.
\end{teor}

A minimal cover of $G = A_5 \wr C_2$ is given by its socle, $\soc(G) = A_5 \times A_5$, together with the subgroups of the form $N_G(M \times M^a)$ where $a \in A_5$ and $M$ is either the stabilizer of $j \in \{1,2,3,4,5\}-\{i\}$ (for some $i \in \{1,2,3,4,5\}$) in $A_5$ or the normalizer of a Sylow $5$-subgroup of $A_5$.

The lower bounds for the covering number will be obtained by using the following tool, introduced by Mar\'oti in \cite{maroti}.

\begin{defi}[Definite unbeatability] \label{du}
\label{d1} Let $X$ be a group. Let $\mathcal{H}$ be a set
of proper subgroups of $X$, and let $\Pi \subseteq X$. Suppose
that the following four conditions hold for $\mathcal{H}$ and
$\Pi$.
\begin{enumerate}
\item $\Pi \cap H \neq \emptyset$ for every $H \in \mathcal{H}$;

\item $\Pi \subseteq \bigcup_{H \in \mathcal{H}} H$;

\item $\Pi \cap H_{1} \cap H_{2} = \emptyset$ for every distinct
pair of subgroups $H_{1}$ and $H_{2}$ of $\mathcal{H}$;

\item $|\Pi \cap K| \leq |\Pi \cap H|$ for every $H \in
\mathcal{H}$ and $K < X$ with $K \not \in \mathcal{H}$.
\end{enumerate}
Then $\mathcal{H}$ is said to be definitely unbeatable on $\Pi$.
\end{defi}

For $\Pi \subseteq X$ let $\sigma_X(\Pi)$ be the least cardinality
of a family of proper subgroups of $X$ whose union contains $\Pi$.
The following lemma is straightforward.

\begin{lemma}
\label{l6} If $\mathcal{H}$ is definitely unbeatable on $\Pi$ then
$\sigma_X(\Pi)=|\mathcal{H}|$.
\end{lemma}

It follows that if $\mathcal{H}$ is definitely unbeatable on $\Pi$
then $|\mathcal{H}| = \sigma_X(\Pi) \leq \sigma(X)$.

Let us give \cite[Theorem 3.1]{maroti} as an example. Let $n \geq 11$ be an odd integer, and let $X := \Sym(n)$ be the symmetric group on $n$ letters. Let $\mathcal{H}$ be the family of subgroups of $\Sym(n)$ consisting of the alternating group $\Alt(n)$ and the intransitive maximal subgroups of $\Sym(n)$. Let $\Pi$ be the subset of $\Sym(n)$ consisting of the permutations which are product of at most two disjoint cycles. Then $\mathcal{H}$ is a cover of $\Sym(n)$ which is definitely unbeatable on $\Pi$, therefore $\sigma(\Sym(n)) = |\mathcal{H}| = 2^{n-1}$.

This example was rivisited and generalized by Mar\'oti and me (cf. \cite{margar}, \cite{gar2}) and the results summarized in Theorems \ref{t1} and \ref{t2} below were obtained.

Let us fix some notations we will often use.

\begin{nota} \label{mono}
Let $G$ be a monolithic group with socle $N = \soc(G) = S_1 \times \cdots \times S_m$, where $S_1, \ldots ,S_m$ are pairwise isomorphic non-abelian simple groups. $X := N_G(S_1)/C_G(S_1)$ is an almost-simple group with socle $S := S_1 C_G(S_1)/C_G(S_1) \cong S_1$. The minimal normal subgroups of $S^m = S_1 \times \ldots \times S_m$ are precisely its factors, $S_1,\ldots,S_m$. Since automorphisms send minimal normal subgroups to minimal normal subgroups, it follows that $G$ acts on the $m$ factors of $N$. Let $\rho: G \to \Sym(m)$ be the homomorphism induced by the conjugation action of $G$ on the set $\{S_1, \ldots,S_m\}$. $K := \rho(G)$ is a transitive permutation group of degree $m$. By \cite[Remark 1.1.40.13]{spagn} $G$ embeds in the wreath product $X \wr K$. Let $L$ be the subgroup of $X$ generated by the following set: $$S \cup \{x_1 \cdots x_m\ |\ \exists k \in K:\ (x_1, \ldots, x_m) k \in G\}.$$Let $T$ be a normal subgroup of $X$ containing $S$ and contained in $L$ with the property that $L/T$ has prime order if $L \neq S$, and $T=L$ if $L=S$.
\end{nota}

Let $G$ be a primitive monolithic group with non-abelian socle $N$, and write $N=S^m$ with $S$ a non-abelian simple group. The covers of $G$ we often look at consist of some subgroups of $G$ containing $N$ and subgroups of the form $N_G(M \times M^{a_2} \times \cdots \times M^{a_m})$ with $M < S$, which will be called ``\textit{product type subgroups}''.

In the following if $n$ is a positive integer we denote by $\omega(n)$ the number of prime divisors of $n$. Suppose that $G/N$ is cyclic. The covers of $G$ we consider consist of all the $\omega(|G/N|)$ maximal subgroups of $G$ containing $N$ and some product type subgroups $N_G((S \cap M) \times (S \cap M)^{a_2} \times \cdots \times (S \cap M)^{a_m})$ where $a_1=1,a_2,\ldots,a_m \in S$ and $M$ varies in a family of maximal subgroups of $X$ supplementing $S$ which covers a coset $xS$ of $S$ in $X$ which generates the cyclic group $X/S$. This is how we obtain upper bounds for $\sigma(G)$ (the size of a cover of $G$ is an upper bound for $\sigma(G)$).

\begin{teor}[Mar\'oti A., Garonzi M. 2010 \cite{margar}] \label{t1}

Let $G$ be a monolithic group with non-abelian socle, and let us use Notations \ref{mono}. Suppose that $G/N$ is cyclic and that $X = S = \Alt(n)$. Then the following holds.

\begin{enumerate}

\item If $12 < n \equiv 2 \mod(4)$ then $$\sigma(G) = \omega(m) + \sum_{i=1,\ i\ \text{odd}}^{(n/2)-2} \binom{n}{i}^m + \frac{1}{2^m} \binom{n}{n/2}^m.$$

\item If $12 < n \not \equiv 2 \mod(4)$ then $$\omega(m) + \frac{1}{2} \sum_{i=1,\ i\ \text{odd}}^n \binom{n}{i}^m \leq \sigma(G).$$

\item Suppose $n$ has a prime divisor at most $\sqrt[3]{n}$. Then $$\sigma(G) \sim \omega(m) + \min_{\mathcal{M}} \sum_{M \in \mathcal{M}} |S:M|^{m-1}\ \text{as}\ n \to \infty.$$

\end{enumerate}

\end{teor}

\begin{teor}[Garonzi M. 2011 \cite{gar2}] \label{t2}

Let $G$ be a monolithic group with non-abelian socle, and let us use Notations \ref{mono}. Suppose that $G/N$ is cyclic and that $X = \Sym(n)$. Then the following holds.

\begin{enumerate}

\item Suppose that $n \geq 7$ is odd and $(n,m) \neq (9,1)$. Then $$\sigma(G) = \omega(2m) + \sum_{i=1}^{(n-1)/2} \binom{n}{i}^m.$$

\item Suppose that $n \geq 8$ is even. Then $$\left( \frac{1}{2} \binom{n}{n/2} \right)^m \leq \sigma(G) \leq \omega(2m) + \left( \frac{1}{2} \binom{n}{n/2} \right)^m + \sum_{i=1}^{[n/3]} \binom{n}{i}^m.$$In particular $\sigma(G) \sim \left( \frac{1}{2} \binom{n}{n/2} \right)^m$ as $n \to \infty$.

\end{enumerate}

\end{teor}

\section{Attacking the conjecture}

The content of this section is included in my Ph.D. thesis.

The following result provides a first partial reduction to monolithic groups.

\begin{prop} \label{starmin}
Let $H$ be a non-abelian $\sigma$-elementary group, let $N_1,\ldots,N_{\ell}$ be minimal normal subgroups of $H$ such that $\soc(H) = N_1 \times \cdots \times N_{\ell}$ and let $X_1,\ldots,X_{\ell}$ be the primitive monolithic groups associated to $N_1,\ldots,N_{\ell}$ respectively. Then at most one of $N_1,\ldots,N_{\ell}$ is abelian. Suppose that $N_1$ is non-abelian and that $\sigma^{\ast}(X_1) \leq \sigma^{\ast}(X_j)$ whenever $j \in \{1,\ldots,\ell\}$ and $N_j$ is non-abelian. If $\sigma(X_1) < 2 \sigma^{\ast}(X_1)$ then $H \cong X_1$, i.e. $H$ is monolithic.
\end{prop}

\begin{proof}
By Proposition \ref{sigmastar} $$\sigma^{\ast}(X_1) + \sum_{j=2}^{\ell} \sigma^{\ast}(X_j) \leq \sigma(H) \leq \sigma(X_1) < 2 \sigma^{\ast}(X_1).$$It follows that $\sum_{j=2}^{\ell} \sigma^{\ast}(X_j) < \sigma^{\ast}(X_1)$ hence, by the minimality hypothesis on $X_1$, $N_2,\ldots,N_{\ell}$ are abelian. In \cite[Corollary 14]{cubo} it is proved that any non-abelian $\sigma$-elementary group has at most one abelian minimal normal subgroup, thus $\ell = 2$. Since $N_2$ is abelian $\ell_{X_2}(N_2)=|N_2|$, and by Proposition \ref{sigmastar}
\begin{eqnarray}
\min \{2 \sigma^{\ast}(X_1), 2|N_2|\} & \leq & \sigma^{\ast}(X_1) + |N_2|  =  \sigma^{\ast}(X_1) + \ell_{X_2}(N_2) \leq \nonumber \\ & \leq & \sigma(H)  \leq  \min \{\sigma(X_1),\sigma(X_2)\}. \nonumber
\end{eqnarray}
Now by hypothesis $\sigma(X_1) < 2 \sigma^{\ast}(X_1)$, and $\sigma(X_2) < 2|N_2|$ by Proposition \ref{abmns}. This leads to a contradiction.
\end{proof}

In order to prove an inequality like $\sigma(G) < 2 \sigma^{\ast}(G)$ for $G$ a primitive monolithic group we first need some way to get as much general as possible upper bounds for $\sigma(G)$.

\begin{teor} \label{sopra}
Let $G$ be a monolithic group with non-abelian socle, and let us use Notations \ref{mono}. Assume that $X/S$ is abelian. Let $\mathcal{M}$ be a set of maximal subgroups of $X$ supplementing $S$ and such that $\bigcup_{M \in \mathcal{M}} M$ contains a coset $xS \in L$ with the property that $\langle x,T \rangle = L$.

Then $\sigma(G) \leq 2^{m-1} + \sum_{M \in \mathcal{M}} |S:S \cap M|^{m-1}$.
\end{teor}

Unfortunately the hypothesis ``$X/S$ abelian'' does not seem easy to bypass.

\begin{proof}
If $L \neq T$ define $$R := \{(x_1, \ldots ,x_m)k \in G\ |\ x_1 \cdots x_m \in T\}.$$Since $X/S$ is abelian, $R$ is a proper subgroup of $G$.

Let $\delta \in K$ be an $m$-cycle, $1 = a_1, a_2, \ldots ,a_m \in X$ and $M \in \mathcal{M}$. An element $(x_1, \ldots ,x_m) \delta \in X \wr K$ normalizes $(M \cap S) \times (M \cap S)^{a_2} \times \cdots \times (M \cap S)^{a_m}$ if and only if $$(M \cap S)^{a_{\delta^{-1}(1)} x_{\delta^{-1}(1)}} \times (M \cap S)^{a_{\delta^{-1}(2)} x_{\delta^{-1}(2)}} \times \cdots \times (M \cap S)^{a_{\delta^{-1}(m)} x_{\delta^{-1}(m)}} =$$ $$= (M \cap S) \times (M \cap S)^{a_2} \times \cdots \times (M \cap S)^{a_m},$$if and only if
\begin{equation} \label{prodeq}
a_{\delta^{-1}(1)} x_{\delta^{-1}(1)} a_1^{-1},\ a_{\delta^{-1}(2)} x_{\delta^{-1}(2)} a_2^{-1}, \ldots, a_{\delta^{-1}(m)} x_{\delta^{-1}(m)} a_m^{-1} \in N_X(M \cap S) = M.
\end{equation}
If $x_1 x_{\delta(1)} \cdots x_{\delta^{m-1}(1)} \in M$ then there exist $a_2, \ldots ,a_m \in X$ such that (\ref{prodeq}) is true. Since $M$ supplements $S$ in $X$, $a_2, \ldots ,a_m$ can be chosen in $S$. Therefore every element $(x_1, \ldots, x_m) \delta \in G$ such that $\delta$ is an $m$-cycle and $x_1 x_{\delta(1)} \cdots x_{\delta^{m-1}(1)} \in xS$ belongs to a subgroup of $G$ of the form $N_G((M \cap S) \times (M \cap S)^{a_2} \times \cdots \times (M \cap S)^{a_m})$ where $M \in \mathcal{M}$ and $a_2, \ldots, a_m \in S$. It follows that $G$ is covered by these subgroups together with $R$ (if $L \neq T$) and the pre-images through $\rho$ of $2^{m-1}-1$ maximal intransitive subgroups of $K$ (corresponding to the subsets of $\{1,\ldots,m\}$ of size from $1$ to $[m/2]$).
\end{proof}

Recall the structure of maximal subgroups of primitive monolithic groups.

\begin{defi}[\cite{spagn}, Definition 1.1.37]
Let $G = \prod_{i=1}^n S_i$ be a direct product of groups. A subgroup $H$ of $G$ is said to be ``full diagonal'' if each projection $\pi_i: H \to S_i$ is an isomorphism.
\end{defi}

What follows is part of the O'Nan-Scott theorem (reference: \cite[Remark 1.1.40]{spagn}). Let $G$ be a primitive monolithic group with non-abelian socle $N=S^m$. Let $H$ be a maximal subgroup of $G$ such that $N \not \subseteq H$, i.e. $HN=G$, i.e. $H$ supplements $N$. Suppose $N \cap H \neq \{1\}$, i.e. $H$ does not complement $N$. Since $N$ is the unique minimal normal subgroup of $G$ and $H$ is a maximal subgroup of $G$ not containing $N$, $H=N_G(N \cap H)$. In the following let $X := N_G(S_1)/C_G(S_1)$ (it is an almost simple group with socle $S_1C_G(S_1)/C_G(S_1) \cong S$). There are two possibilities for the intersection $N \cap H$:

\begin{enumerate}

\item \textbf{Product type}. Suppose the projections $H \to S_i$ are not surjective. Then there exists a subgroup $M$ of $S$ such that $N_X(M)$ supplements $S$ in $X$ and elements $a_2,\ldots,a_m \in S$ such that $$H \cap N = M \times M^{a_2} \times \ldots \times M^{a_m}.$$In this case $|H \cap N| = |M|^m$.

\item \textbf{Diagonal type}. Suppose the projections $H \to S_i$ are surjective. Then there exists an $H$-invariant partition $\Delta$ of $\{1,\ldots,m\}$ into blocks for the action of $H$ on $\{1,\ldots,m\}$ such that $$H \cap N = \prod_{D \in \Delta} (H \cap N)^{\pi_D}$$and for each $D \in \Delta$ the projection $(H \cap N)^{\pi_D}$ is a full diagonal subgroup of $\prod_{i \in D}S_i$. In this case $|H \cap N| \leq |S|^{m/r}$ where $r$ is the smallest prime divisor of $m$.

\end{enumerate}

We now prove a crucial lemma which we will need in the proof of the main theorem. Let $G$ be a monolithic group with non-abelian socle, and let us use Notations \ref{mono}.

Let $\mathcal{Z}$ be the set of pairs $(z,w)$ in $X \times X$ such that $\langle z^a, w^b \rangle \supseteq S$ for every $a,b \in S$. By \cite{kls}, $\mathcal{Z} \cap (S \times S) \neq \emptyset$. Let $k$ be a non-$m$-cycle in $K$, let $O_1 = (i_1, \ldots ,i_r)$, $O_2 = (j_1, \ldots ,j_s)$ be two cycles in the cyclic decomposition of $k$, and for $\rho^{-1}(k) \ni h = (x_1, \ldots ,x_m) k$, with $x_1,\ldots,x_m \in X$, let $h_{O_1} := x_{i_1} \cdots x_{i_r}$ and $h_{O_2} := x_{j_1} \cdots x_{j_s}$.

\begin{lemma} \label{nonciclo}
Let $\mathcal{E}_k := \{(h_{O_1},h_{O_2})\ |\ h \in \rho^{-1}(k)\} \cap \mathcal{Z}$. Let $r$ be the smallest prime divisor of $m$. If $g \in \rho^{-1}(k)$ then $\sigma_{Ng}(G) \geq |\mathcal{E}_k| \cdot |S|^{m-m/r-2}$.
\end{lemma}

\begin{proof}
Let $$\mathfrak{X} := \{h \in Ng\ |\ (h_{O_1}, h_{O_2}) \in \mathcal{E}_k \}.$$ Note that if $h \in Ng$, $\theta$, $\varphi \in X$ are such that $h_{O_1} \equiv \theta \mod S$ and $h_{O_2} \equiv \varphi \mod S$ then there exists $t \in N$ such that $(th)_{O_1} = \theta$, $(th)_{O_2} = \varphi$. This implies that $|\mathfrak{X}| \geq |\mathcal{E}_k| \cdot |S|^{m-2}$. It is easy to show that if $a_2, \ldots ,a_m \in S$ and $h \in \rho^{-1}(k) \cap N_G(M \times M^{a_2} \times \cdots \times M^{a_m})$ then $h_{O_1} \in N_X(M)^{a_{i_1}}$, $h_{O_2} \in N_X(M)^{a_{j_1}}$. By the definition of $\mathcal{E}_k$, we deduce that $\mathfrak{X} \cap H = \emptyset$ whenever $H$ is a supplement of $N$ of product type. Since the maximal subgroups of $G$ complementing $N$ intersect $Ng$ in at most one point, this implies that in order to cover $\mathfrak{X}$ with supplements of $N$ we need at least $|\mathcal{E}_k| \cdot |S|^{m-2}/|S|^{m/r}$ of them.
\end{proof}

We are ready to state the main theorem.

\begin{teor}
Let $H$ be a $\sigma$-elementary non-abelian group. We will use the notations of Theorem \ref{struttura}. Let $N = N_1$ be a non-abelian minimal normal subgroup of $H$ and let $G := H/R_H(N) = X_1$ be the primitive monolithic group associated to $N$. Assume that $\min \{\sigma^{\ast}(X_i)\ :\ i=1,\ldots,h\} = \sigma^{\ast}(G)$. Let us use Notations \ref{mono}, and let $r$ be the smallest prime divisor of $m$. Suppose that $X/S$ is abelian. Let $E_{nc} := \min \{|\mathcal{E}_k|\ |\ k \in K\ \text{non-}m\text{-cycle}\}$ ($\mathcal{E}_k$ is as in Lemma \ref{nonciclo}). Suppose that whenever $x \in X$ is such that $\langle x,S \rangle = T$ there exist families $\mathcal{M},\mathcal{J}$ of maximal subgroups of $X$ supplementing $S$ such that:

\begin{enumerate}
\item $xS \subseteq \bigcup_{M \in \mathcal{M} \cup \mathcal{J}} M$;
\item $\sum_{M \in \mathcal{M} \cup \mathcal{J}} |S:S \cap M|^{m-1} < E_{nc} \cdot |S|^{m-m/r-2}$;
\item $\sigma_{Ny}(Y) \geq \sum_{M \in \mathcal{M}} |S:S \cap M|^{m-1}$ (notation is as in Definition \ref{star}) whenever $Y$ is a primitive monolithic group with socle $N$ and $y \in Y$ is such that $\langle N,y \rangle = Y$.
\item $\sum_{M \in \mathcal{J}} |S:S \cap M|^{m-1} + 2^{m-1} < \sum_{M \in \mathcal{M}} |S:S \cap M|^{m-1}$;
\end{enumerate}

Then $H \cong G$, in other words $H$ is monolithic.

\end{teor}

\begin{proof}
By Lemma \ref{starmin}, it is enough to show that $\sigma(G) < 2 \sigma^{\ast}(G)$. Let us do that. Since (1) holds, we may apply Theorem \ref{sopra} and obtain that $\sigma(G) \leq 2^{m-1} + \sum_{M \in \mathcal{M} \cup \mathcal{J}} |S:S \cap M|^{m-1}$. Fix a set $\Omega$ of cosets of $N$ in $G$ such that $\sigma^{\ast}(G) = \sigma_{\Omega}(G)$. Clearly, if $gN \in \Omega$ then $\sigma^{\ast}(G) \geq \sigma_{Ng}(G) = \sigma_{Ng}( \langle N,g \rangle)$. By (2) and Lemma \ref{nonciclo}, $\rho(g)$ is an $m$-cycle, therefore $\langle N,g \rangle$ is a primitive monolithic group hence by (3) $$\sigma^{\ast}(G) \geq \sigma_{Ng} (G) = \sigma_{Ng} (\langle N,g \rangle) \geq \sum_{M \in \mathcal{M}} |S:S \cap M|^{m-1}.$$Therefore by Theorem \ref{sopra} and (4), $$\sigma(G) \leq 2^{m-1} + \sum_{M \in \mathcal{M} \cup \mathcal{J}} |S:S \cap M|^{m-1} < 2 \sum_{M \in \mathcal{M}} |S:S \cap M|^{m-1} \leq 2 \sigma^{\ast}(G).$$Therefore $\sigma(G) < 2 \sigma^{\ast}(G)$.
\end{proof}

Fulfilling condition (3) requires the type of results listed in Theorems \ref{t1} and \ref{t2} (indeed, note that in Condition (3) the quotient $Y/N$ is cyclic). In my Ph.D. thesis I give several examples of applications of this result, and in particular I prove the following.

\begin{cor}
Let $H$ be a non-abelian $\sigma$-elementary group. If all the minimal subnormal subgroups of $H$ are alternating groups of even degree larger than $30$ then $H$ is monolithic.
\end{cor}

\end{document}